\documentclass[11pt]{amsart}

\usepackage[english]{babel}
\usepackage[utf8]{inputenc}
\usepackage[T1]{fontenc}
\IfFileExists{lmodern.sty}{\usepackage{lmodern}}{}

\usepackage[letterpaper,top=2.4cm,bottom=2.4cm,left=2.4cm,right=2.4cm]{geometry}

\usepackage{amssymb,amsfonts,amsmath,amsthm,mathrsfs}
\usepackage[shortlabels]{enumitem}
\usepackage[all]{xy}
\usepackage[dvipsnames]{xcolor}
\usepackage{hyperref}
\hypersetup{
  colorlinks = true,
  linkcolor  = blue,
  urlcolor   = violet,
  citecolor  = Green}

\theoremstyle{plain}
\newtheorem{Theorem}{Theorem}[section]
\newtheorem{Proposition}[Theorem]{Proposition}
\newtheorem{Lemma}[Theorem]{Lemma}

\theoremstyle{definition}
\newtheorem{Definition}[Theorem]{Definition}
\newtheorem{Example}[Theorem]{Example}

\theoremstyle{remark}
\newtheorem{Remark}[Theorem]{Remark}

\newcommand{\K}{\mathcal{K}}
\newcommand{\F}{\mathcal{F}}
\newcommand{\U}{\mathbb{U}}
\newcommand{\Jet}{\mathcal{J}}
\newcommand{\Xf}{\mathfrak{X}}
\newcommand{\Lpt}{\mathcal{L}_{\mathrm{pt}}}
\newcommand{\Lom}{\mathcal{L}_{\Omega}}
\newcommand{\Symop}{\mathrm{Sym}_{\mathrm{op}}}
\newcommand{\Sym}{\mathbf{Sym}}
\newcommand{\SymOm}{\mathrm{Sym}_{\Omega}}
\newcommand{\Vertt}{\mathrm{Vert}}
\newcommand{\prol}{\mathrm{prol}}
\newcommand{\dx}{\partial_x}
\newcommand{\dy}{\partial_y}
\newcommand{\dyp}{\partial_{y'}}
\DeclareMathOperator{\Wr}{W}

\title[Several notions of symmetry
       for the second order linear differential equation]{On the relations between several notions of symmetry
       for the second order linear differential equation}

\author[D. Bl\'azquez-Sanz]{David Bl\'azquez-Sanz}
\address{David Bl\'azquez-Sanz, Universidad Nacional de Colombia - Sede Medell\'in, Facultad de Ciencias, Departamento de Matem\'aticas, Colombia}
\email{dblazquezsa@unal.edu.co}

\author[S. A. Aguirre Agudelo]{Santiago Alexis Aguirre Agudelo}
\address{Santiago Alexis Aguirre Agudelo}
\email{saaguirrea@unal.edu.co}

\date{\today}

\begin{document}
\maketitle

\begin{abstract}
There are several non-equivalent notions of infinitesimal symmetry in the
literature of second order linear differential equations: Lie
point symmetries, vertical (gauge) symmetries, operator symmetries,
infinitesimal contact symmetries, and Lie--B\"acklund operators. We construct
an explicit correspondence among the first three,
We then describe the Lie algebra $\Lom(\U)$ of infinitesimal contact
transformations of the contact system $\Omega=\langle dy-y'\,dx\rangle$, with
coefficients in a differential field $\U$ of functions of $x$ and $y$. We
obtain a canonical decomposition $\Lom(\U)=\prod_{k\ge0}\Lom^{k}(\U)$ into
$\mathbb C$-vector spaces, each parametrized by $\U$ (by $\U\oplus\U$ for
$k=0$), and we compute the algebraic differential formulae for the Lie bracket in these coordinates. 
Applied to the symmetry problem, we prove that a contact vector field with generating
function $W$ is a symmetry of if and only if
$A^{2}W=aW+b\,AW$, where $A$ is the vector field in the jet space corresponding to the equation; equivalently, if and
only if $W=F_1(u_1,u_2)\phi_1+F_2(u_1,u_2)\phi_2$ for arbitrary functions
$F_1,F_2$ of the two first integrals of $A$ and a basis $\phi_1,\phi_2$ of
solutions. The symmetry
algebra is always parametrized by two arbitrary functions of two variables. It
also shows that the decomposition of $\Lom(\U)$ never captures the whole
symmetry algebra: for $a\neq0$ the graded part reduces to the point
symmetries, while for $y''=0$ it is an infinite dimensional but still
\emph{proper} subspace, and in neither case is it a Lie subalgebra. Finally we
make precise the transformation law $W\mapsto\mu^{-1}(W\circ\varphi)$ for
characteristics under a contact transformation with conformal factor $\mu$,
which governs the transport of evolutionary representatives.
\end{abstract}



% ===========================================================================
\section{Introduction}\label{sec:intro}
% ===========================================================================

Lie's theory of infinitesimal symmetries and the Picard--Vessiot theory of
linear differential equations provide two systematic approaches to
understanding their solutions. The
former is geometric and local, associating to the equation a Lie algebra of
vector fields; the latter is algebraic and global, associating with it a linear
algebraic group. Clarifying how the two are connected is a classical topic that has been explored in \cite{AT,OV, BMW}. Indeed, there are several non-equivalent notions of infinitesimal symmetries. In this paper, we elucidate the relationships among these notions and investigate how contact symmetries interact with the Picard--Vessiot framework.

Consider the second order homogeneous linear differential equation
\begin{equation}\label{LE}
  y'' = a(x)\,y + b(x)\,y' .
\end{equation}
At least five non-equivalent notions of infinitesimal symmetry occur in the
literature:
\begin{enumerate}[(a)]
\item \emph{Infinitesimal Lie point symmetries}: vector fields on the
  $(x,y)$-plane whose flow preserves the solution curves of \eqref{LE}. The
  corresponding Lie algebra has dimension $8$ and is isomorphic to
  $\mathfrak{sl}_3(\mathbb C)$ \cite{IB,IbragimovAnderson}.
\item \emph{Vertical symmetries} in the sense of \cite{BMW}: vertical vector
  fields on the vector bundle $\mathcal U\times\mathbb C^{2}\to\mathcal U$
  preserving the $1$-jet lifts of solutions. These form an infinite
  dimensional Lie algebra, graded by degree in $(y,y')$.
\item \emph{Operator symmetries} in the sense of \cite{AT,OV}, whose
  Galois-theoretic role is explored in \cite{IBB,BL}: first order linear
  differential operators mapping solutions of \eqref{LE} to solutions. These
  form a $4$-dimensional Lie algebra.
\item \emph{Infinitesimal contact symmetries}: contact vector fields on the
  $3$-dimensional jet space $(x,y,y')$ whose flow leaves invariant the contact
  lifts of solutions.
\item \emph{Lie--B\"acklund operators} \cite{IbragimovAnderson}: evolutionary
  vector fields whose characteristics may depend on derivatives of arbitrary
  order.
\end{enumerate}

We first construct an explicit correspondence among (a), (b) and (c). We then
analyse the contact symmetry algebra (d) in relation with Lie--B\"acklund operators and prove a description of it that is
uniform in $a$ and $b$, from which the behaviour of the natural decomposition
of $\Lom(\U)$ follows in all cases. Finally we make precise the transformation
law for characteristics that governs the transport of evolutionary
representatives, and use it to settle a point of interpretation in the
literature.

\subsection*{Main results}

Let $A=\dx+y'\dy+(ay+by')\dyp$ be the vector field on the $1$-jet space
$J=\Jet^{1}(\mathbb C,\mathbb C)$ associated with \eqref{LE}, and let
$L=\dx^{2}-b\dx-a$ be the associated differential operator.

\begin{enumerate}[(i)]
\item \textbf{Dictionary between (a), (b) and (c).} If
  $X=\xi\dx+\eta\dy\in\Lpt(A)$, then
  $\Phi(X):=-\xi(x,0)\dx+\eta_y(x,0)\in\Symop(L)$
  (Theorem~\ref{thm:pt-to-op}); conversely, $P\in\Symop(L)$ induces the linear
  vertical symmetry $\Psi^{-1}(P)=(P\cdot y)\dy+(\dx P\cdot y)\dyp\in\Sym_A^1$
  (Theorem~\ref{thm:op-to-pt}). In Theorem~\ref{thm:diagram} we prove that the
  resulting diagram commutes and that $\Psi\colon\Sym_A^1\to\Symop(L)$ is an
  \emph{anti}-isomorphism of Lie algebras, while $\Phi$ is only
  $\mathbb C$-linear (Remark~\ref{rem:Phi-not-morphism}).
\item \textbf{Structure of $\Lom(\U)$.} Infinitesimal contact transformations
  with coefficients in $\U[[y']]$ decompose uniquely as $X=\sum_{k\ge0}X_k$
  with $X_k\in\Lom^{k}(\U)$ (Proposition~\ref{prop:decomposition}). The
  decomposition is the expansion in powers of $y'$ of the generating function
  (Proposition~\ref{prop:char-of-components}), and this reduces the
  computation of the bracket to a monomial identity
  (Lemma~\ref{lem:monomial}), giving Theorem~\ref{thm:bracket} with a short
  proof. The summand $\Lom^{0}(\U)\cong\Xf(\mathbb C^{2},\U)$ is the unique
  Lie subalgebra of the decomposition.
\item \textbf{A uniform description of the contact symmetries.} A contact
  field $X_W$ is a symmetry of \eqref{LE} if and only if
  \[
    A^{2}W \;=\; a\,W + b\,AW
  \]
  (Theorem~\ref{thm:global}); equivalently, if and only if
  $W=F_1(u_1,u_2)\phi_1+F_2(u_1,u_2)\phi_2$ where $u_1,u_2$ are the first
  integrals of $A$, $\phi_1,\phi_2$ a basis of solutions, and $F_1,F_2$ are
  arbitrary (Theorem~\ref{thm:firstintegrals}). In particular
  $\SymOm(A)\cong\SymOm(A_0)$ for every $a,b$.
\item \textbf{The decomposition never captures the symmetry algebra.} If
  $a\neq0$ then $\SymOm^{k}(A)=0$ for all $k\ge1$
  (Theorem~\ref{thm:general-collapse}). For $y''=0$ the graded pieces are
  nonzero, of dimensions $8,3,3,3,\dots$, but the sum is still a
  \emph{proper} subspace of $\SymOm(A_0)$ and is not a Lie subalgebra
  (Theorem~\ref{thm:trivial}). The mechanism is identified in
  Remark~\ref{rem:grading-broken}: the contact transformations relating the
  various equations introduce binomial mixing in the powers of $y'$.
\item \textbf{Rationality and the Galois action.} The symmetries with
  polynomial characteristic have coefficients in $\F[y,y']$, where $\F$ is the
  Picard--Vessiot extension of \eqref{LE}, and the differential Galois group
  acts on $\SymOm(A)$ through its action on $(\phi_1,\phi_2)$ and $(u_1,u_2)$
  (Theorem~\ref{thm:rationality}). Without the polynomiality hypothesis the
  inclusion fails (Example~\ref{ex:nonpolynomial}).
\item \textbf{Characteristics of contact and Lie--B\"acklund fields.} Under a
  contact transformation with conformal factor $\mu$, characteristics
  transform as $W\mapsto\mu^{-1}(W\circ\varphi)$
  (Proposition~\ref{prop:char-transform}). This distinction between vector
  field pushforward and characteristic scaling resolves a subtlety in the
  literature and confirms the validity of the tables of Martini and Kersten
  \cite{Phys} (Remark~\ref{rem:Phys}).
\end{enumerate}

% \subsection*{On the computations}

% Several of the identities below are long, and some of them have been stated
% incorrectly in earlier accounts. All displayed identities in this paper have
% been verified independently with a computer algebra system: the determining
% equations \eqref{eqf}--\eqref{eqh} and
% \eqref{sympo_a}--\eqref{symop_b}, the verticalization
% \eqref{eq:XV0}--\eqref{eq:XV2}, the bracket formulas
% \eqref{eq:bracket0k}--\eqref{eq:bracketkl}, the determining equations of
% Sections~\ref{sec:contact-sym}, and the transported fields
% \eqref{eq:Xk-transformed}. Where a general index $k$ occurs the identity was
% checked for $k\le4$, and for the double-index bracket for $1\le k,l\le3$.

\subsection*{Acknowledgements}

This paper is based on the master's thesis of the first author \cite{Aguirre2025},
written under the supervision of the second author at Universidad Nacional de
Colombia, Sede Medell\'in. The second author acknowledges the support of the
research project ``Algunas conexiones de la teor\'ia de Picard-Vessiot con la
geometr\'ia y la din\'amica'' of Universidad Nacional de Colombia, code Hermes
67312.

\subsection*{Use of AI-Assisted Tools}

During the preparation of this work, the authors employed Claude Opus (Anthropic) and Gemini 3.7 (Google) to audit and verify the results from \cite{Aguirre2025}. Insights from these tools led to the correction of minor inaccuracies, the restructuring of select proofs, and the refinement of theoretical statements as shown in the actual version of this article. All final calculations were independently verified via SymPy-assisted scripts executed through Claude Opus.

% ===========================================================================
\section{Setting and notation}\label{sec:setting}
% ===========================================================================

\subsection{Differential fields}\label{subsec:fields}

Throughout, $\mathcal U\subset\mathbb C$ is a connected open domain and
$(\K,\dx)$ is a differential field of meromorphic functions on $\mathcal U$
containing $\mathbb C(x)$, with field of constants $\mathbb C$
\cite{Seidenberg}. We fix $a,b\in\K$, a point $x_0\in\mathcal U$ which is a
regular point of $a$ and $b$, and a simply connected neighbourhood
$\mathcal U'\subset\mathcal U$ of $x_0$ on which $a$ and $b$ are holomorphic.
All statements below about dimensions of solution spaces of linear equations
refer to this domain, where existence and uniqueness hold and the solution
space of an equation of order $n$ has dimension $n$ over $\mathbb C$.

Let $\phi_1,\phi_2$ be a basis of solutions of \eqref{LE} on $\mathcal U'$.
The Picard--Vessiot extension of $\K$ is
$(\F,\dx):=\K\langle\phi_1,\phi_2\rangle$ \cite{SV,TC}. We adopt the Wronskian
convention
\begin{equation}\label{eq:wronskian}
  \Wr(u,v) \;=\; \det\begin{pmatrix} u & v\\ u' & v'\end{pmatrix}
  \;=\; uv'-vu',
\end{equation}
which satisfies $\Wr'=b\,\Wr$ for solutions of \eqref{LE}, so that $\Wr$ is
nonvanishing as soon as it is nonzero at one point.

We denote by $\U$ a differential field of functions of two variables $(x,y)$,
with commuting derivations $\dx,\dy$, containing $\F$ in its field of
$\dy$-constants. In the applications $\U=\F(y)$.

\subsection{The $1$-jet space and the contact system}\label{subsec:jets}

Let $J:=\Jet^{1}(\mathbb C,\mathbb C)|_{\mathcal U}\simeq\mathcal
U\times\mathbb C^{2}$ with coordinates $(x,y,y')$, where $x$ is the base
coordinate and $(y,y')$ are fibre coordinates
\cite{BM,ContacGeometry,GeoDif}. The contact distribution $\Omega$ is
generated by the contact form
\begin{equation}\label{eq:contact-form}
  \omega \;=\; dy - y'\,dx .
\end{equation}
A holomorphic function $u\colon\mathcal U'\to\mathbb C$ lifts to its $1$-jet
prolongation $j^{1}u(x)=(x,u(x),u'(x))$, whose tangent vectors annihilate
$\omega$.

\subsection{Encoding the equation}\label{subsec:eq-two-ways}

Equation \eqref{LE} can be expressed in two equivalent ways.
\begin{enumerate}
\item Geometrically, \emph{as a vector field on $J$}:
\begin{equation}\label{eq:A}
  A \;=\; \dx + y'\,\dy + \bigl(a(x)y+b(x)y'\bigr)\,\dyp .
\end{equation}
Holomorphic solutions $u$ of \eqref{LE} correspond bijectively to integral
curves $j^{1}u$ of $A$. We write $A_0=\dx+y'\dy$ for $y''=0$, and
$A_1=\dx+y'\dy+\alpha(x)y\,\dyp$ for the conservative equation
$y''=\alpha(x)y$, $\alpha\in\K$.
\item Algebraically, \emph{as an operator in $\K[\dx]$}:
\begin{equation}\label{eq:L}
  L \;=\; \dx^{2}-b\,\dx-a \;\in\;\K[\dx],\qquad \ker L\subset\F .
\end{equation}
\end{enumerate}
A vector field $Y\in\Xf(J)$ is \emph{vertical} if $Y(x)=0$. For any $Y$ with
$Y(x)\neq0$, its vertical representative is
\begin{equation}\label{eq:verticalmap}
  \Vertt(Y) \;:=\; Y-(Yx)\,A .
\end{equation}
Note that $\Vertt$ is a Lie algebra morphism from the Lie algebra of vector fields in $J$ to that of vertical vector fields on $J$

% ===========================================================================
\section{Classical notions of infinitesimal symmetry}\label{sec:classical}
% ===========================================================================

\subsection{Jet prolongation}\label{subsec:prolongation}

Let $X=\xi(x,y)\dx+\eta(x,y)\dy\in\Xf(\mathcal U\times\mathbb C)$. Its unique
prolongation $\widetilde X=\xi\dx+\eta\dy+\gamma\dyp\in\Xf(J)$ satisfying
$\mathcal L_{\widetilde X}\omega=\lambda\omega$ is given by
\cite{IbragimovAnderson}
\begin{equation}\label{eq:prolongation}
  \gamma \;=\; \eta_x+(\eta_y-\xi_x)\,y'-\xi_y\,y'^{2},
  \qquad \lambda=\eta_y-\xi_y\,y' .
\end{equation}
Prolongation defines an injective morphism of Lie algebras
$\prol\colon\Xf(\mathcal U\times\mathbb C)\to\Xf(J)$.

\subsection{Lie point symmetries}\label{subsec:point}

\begin{Definition}
A vector field $X\in\Xf(\mathcal U\times\mathbb C)$ is a \emph{Lie point
symmetry} of $A$ if $[A,\widetilde X]=\lambda A$ for some function $\lambda$ on
$J$. We denote by $\Lpt(A)$ the Lie algebra of point symmetries.
\end{Definition}

\begin{Theorem}\label{thm:point}
Let $X=\xi(x,y)\dx+\eta(x,y)\dy$. Then $X\in\Lpt(A)$ if and only if
\begin{align*}
  \xi(x,y) &= f(x)\,y+g(x),\\
  \eta(x,y) &= \bigl(f'(x)+b(x)f(x)\bigr)y^{2}+h(x)\,y+k(x),
\end{align*}
where $k$ is a solution of \eqref{LE} and $f,g,h$ satisfy
\begin{align}
  f'' &= af-(bf)', \label{eqf}\\
  g'''+(-4a+2b'-b^{2})\,g'+(-2a'-bb'+b'')\,g &= 0, \label{eqg}\\
  2h' &= g''+(bg)'. \label{eqh}
\end{align}
In particular $\dim_{\mathbb C}\Lpt(A)=2+3+1+2=8$ and
$\Lpt(A)\subset\F[y]\dx+\F[y]\dy$.
\end{Theorem}

\begin{proof}
Substituting \eqref{eq:prolongation} into $[A,\widetilde X]=\lambda A$ and
eliminating $\lambda$ via the $\dx$-component gives a polynomial identity of
degree $3$ in $y'$ in the $\dyp$-component. The coefficients of $y'^{3}$ and
$y'^{2}$ force $\xi_{yy}=0$ and $\eta_{yy}=2\xi_{xy}+2b\xi_y$, yielding the
stated dependence on $y$. The coefficients of $y'^{1}$ and $y'^{0}$, expanded
in powers of $y$, yield \eqref{eqf}, \eqref{eqh}, the equation $L(k)=0$ and,
after eliminating $h$ by \eqref{eqh}, equation \eqref{eqg}. The dimension count
is $2$ for $f$, $3$ for $g$, $1$ for the constant of integration in $h$ and
$2$ for $k$. All four equations are linear with coefficients in $\K$, and
their solutions are obtained from $\phi_1,\phi_2$ by quadratures, whence the
last inclusion.
\end{proof}

\begin{Remark}\label{rem:sl3}
The Lie algebra $\Lpt(A)$ is isomorphic to $\mathfrak{sl}_3(\mathbb C)$. For
$a=b=0$ this is the classical statement that the point symmetries of $y''=0$
are the infinitesimal projective transformations of the plane, the associated
group being $\mathrm{PSL}(3,\mathbb C)$ acting on $\mathbb{CP}^{2}$
\cite[Ch.~2]{IbragimovAnderson}, \cite{IB}. The general case follows because
\eqref{LE} is carried to $y''=0$ by a point transformation
(Section~\ref{subsec:general}), and $\prol$ is a Lie algebra morphism.
\end{Remark}

\subsection{Operator symmetries}\label{subsec:operator}

\begin{Definition}[\cite{AT,OV}]\label{def:opsym}
An operator $P=\beta\dx+\alpha\in\K[\dx]$ is an \emph{operator symmetry} of
$L$ if $LP=ML$ for some $M\in\K[\dx]$. We write $\Symop(L)$ for the set of
operator symmetries.
\end{Definition}

Condition $LP=ML$ says exactly that $P$ maps $\ker L$ into $\ker L$.

\begin{Theorem}\label{thm:opsym}
An operator $P=\beta\dx+\alpha\in\K[\dx]$ belongs to $\Symop(L)$ if and only if
\begin{align}
  2\alpha'+\beta''+(b\beta)' &= 0, \label{sympo_a}\\
  \beta'''+(-4a+2b'-b^{2})\,\beta'+(-2a'-bb'+b'')\,\beta &= 0. \label{symop_b}
\end{align}
Moreover $\Symop(L)$ is a Lie algebra over $\mathbb C$ under the operator
commutator, $\dim_{\mathbb C}\Symop(L)=4$, and $\Symop(L)\subset\F[\dx]$.
\end{Theorem}

\begin{proof}
Comparing degrees and leading coefficients in $LP=ML$ shows that $M$ has order
$1$ with the same leading coefficient as $P$, say $M=\beta\dx+\tilde\alpha$.
Expanding and comparing the coefficients of $\dx^{2},\dx^{1},\dx^{0}$ gives
$\tilde\alpha=\alpha+2\beta'$, then \eqref{sympo_a}, and then
$\alpha''-b\alpha'+(a\beta)'+a\beta'=0$; eliminating $\alpha',\alpha''$ by
\eqref{sympo_a} yields $-\tfrac12$ times the left hand side of
\eqref{symop_b}. Equation \eqref{symop_b} is linear of order $3$, and
\eqref{sympo_a} then determines $\alpha$ up to an additive constant, so
$\dim_{\mathbb C}\Symop(L)=3+1=4$ on the domain $\mathcal U'$ fixed in
Section~\ref{subsec:fields}. If $LP_i=M_iL$ for $i=1,2$ then
$L[P_1,P_2]=[M_1,M_2]L$, so $\Symop(L)$ is closed under the commutator.
\end{proof}

\begin{Remark}\label{rem:match}
Equation \eqref{symop_b} coincides with \eqref{eqg} and \eqref{sympo_a} with
\eqref{eqh} under the substitution $\beta=-g$, $\alpha=h$. This coincidence is
the source of the dictionary of Section~\ref{subsec:dictionary}.
\end{Remark}

\subsection{Vertical symmetries and verticalization}\label{subsec:vertical}

A vertical vector field $Y=f_1(x,y,y')\dy+f_2(x,y,y')\dyp$ is a \emph{vertical
symmetry} of $A$ if $[A,Y]=0$ \cite{BMW}. We say that $Y$ is \emph{polynomial} if its coefficients $f_1$, $f_2$ are polynomials in $y$ and $y'$.
The space $\Sym_A^{<\infty}$ of
polynomial vertical symmetries decomposes into homogeneous parts
$\Sym_A^{<\infty}=\bigoplus_{r\ge0}\Sym_A^{r}$ with respect to $(y,y')$, and
$[\Sym^{r}_A,\Sym^{s}_A]\subset\Sym^{r+s-1}_A$ \cite[Lemma 3.3]{BMW}.

\begin{Proposition}\label{prop:verticalization}
For $X\in\Lpt(A)$, the vertical representative
$X_{\Vertt}:=\Vertt(\widetilde X)$ decomposes as
$X_{\Vertt}=X^{0}_{\Vertt}+X^{1}_{\Vertt}+X^{2}_{\Vertt}$, where each
$X^{r}_{\Vertt}\in\Sym_A^{r}$ is separately a vertical symmetry:
\begin{align}
  X_{\Vertt}^{0} &= k\,\dy+k'\,\dyp, \label{eq:XV0}\\[2pt]
  X_{\Vertt}^{1} &= \bigl(-g\,y'+h\,y\bigr)\dy
    + \bigl(-g'y'-ag\,y-bg\,y'+h'y+h\,y'\bigr)\dyp, \label{eq:XV1}\\[2pt]
  X_{\Vertt}^{2} &= \Bigl((f'+bf)y^{2}-f\,yy'\Bigr)\dy
    + \Bigl((f''+b'f+bf')y^{2}+(f'+bf)yy'-af\,y^{2}-f\,y'^{2}\Bigr)\dyp .
    \label{eq:XV2}
\end{align}
The images of $\Lpt(A)$ in degrees $0,1,2$ have dimensions $2$, $4$ and $2$
respectively, and
\[
  \Vertt\bigl(\Lpt(A)\bigr)^{0}=\Sym_A^{0},\qquad
  \Vertt\bigl(\Lpt(A)\bigr)^{1}=\Sym_A^{1},\qquad
  \Vertt\bigl(\Lpt(A)\bigr)^{2}\subsetneq\Sym_A^{2}.
\]
\end{Proposition}

\begin{proof}
Since $\widetilde X(x)=\xi$, formula \eqref{eq:verticalmap} gives
$X_{\Vertt}=(\eta-\xi y')\dy+(\gamma-\xi(ay+by'))\dyp$; substituting the
expressions of Theorem~\ref{thm:point} and collecting homogeneous parts in
$(y,y')$ gives \eqref{eq:XV0}--\eqref{eq:XV2}. The operator $[A,\,\cdot\,]$
preserves the degree in $(y,y')$, because $A$ is the sum of $\dx$ and a linear
vertical field; hence $0=[A,X_{\Vertt}]=\sum_r[A,X^{r}_{\Vertt}]$ splits into
components of pairwise distinct degrees, forcing $[A,X^{r}_{\Vertt}]=0$ for
every $r$. The dimensions are those of the parameters $(k)$, $(g,h)$ and
$(f)$, that is $2$, $3+1$ and $2$. The last inclusion is strict already for
$y''=0$: $\Sym^{2}_{A_0}$ contains, for instance, quadratic vertical fields
not arising from point symmetries.
\end{proof}

\subsection{The dictionary between point, vertical and operator symmetries}
\label{subsec:dictionary}

Set $P_X:=-g\dx+h$ for $X\in\Lpt(A)$ as in Theorem~\ref{thm:point}. A direct
computation using \eqref{LE} gives
\[
  P_X\cdot y=-gy'+hy,\qquad
  \dx(P_X\cdot y)=-g'y'-agy-bgy'+h'y+hy' ,
\]
so that \eqref{eq:XV1} may be rewritten as
\begin{equation}\label{eq:XV1-operator}
  X^{1}_{\Vertt}=\bigl(P_X\cdot y\bigr)\dy+\bigl(\dx P_X\cdot y\bigr)\dyp .
\end{equation}

\begin{Theorem}\label{thm:pt-to-op}
If $X=\xi(x,y)\dx+\eta(x,y)\dy\in\Lpt(A)$, then
$\Phi(X):=-\xi(x,0)\dx+\eta_y(x,0)\in\Symop(L)$.
\end{Theorem}

\begin{proof}
By Theorem~\ref{thm:point}, $\xi(x,0)=g$ and $\eta_y(x,0)=h$, so
$\Phi(X)=-g\dx+h$. Also by Theorem~\ref{thm:point}, $g$ satisfies \eqref{eqg}
and $h$ satisfies \eqref{eqh}. Setting $\beta=-g$, $\alpha=h$ and using
Remark~\ref{rem:match}, equations \eqref{sympo_a}--\eqref{symop_b} hold, so
Theorem~\ref{thm:opsym} applies.
\end{proof}

\begin{Theorem}\label{thm:op-to-pt}
If $P\in\Symop(L)$, then
$\Psi^{-1}(P):=(P\cdot y)\dy+(\dx P\cdot y)\dyp\in\Sym_A^{1}$.
\end{Theorem}

\begin{proof}
Write $P=\beta\dx+\alpha$ with $\alpha,\beta\in\K$ satisfying
\eqref{sympo_a}--\eqref{symop_b}, and put $g:=-\beta$, $h:=\alpha$, so that
$g,h$ satisfy \eqref{eqg}--\eqref{eqh}. By Theorem~\ref{thm:point} the field
$X=g\dx+hy\,\dy$ (that is, $f=k=0$) is a point symmetry of $A$. Its
verticalization has vanishing components of degrees $0$ and $2$ by
\eqref{eq:XV0} and \eqref{eq:XV2}, so
$\Vertt(\widetilde X)=X^{1}_{\Vertt}\in\Sym^{1}_A$ by
Proposition~\ref{prop:verticalization}, and by \eqref{eq:XV1-operator} this
field is precisely $(P\cdot y)\dy+(\dx P\cdot y)\dyp$.
\end{proof}

\begin{Theorem}\label{thm:diagram}
The diagram of $\mathbb C$-linear maps
\[
\xymatrix@C=1.4cm@R=1.5cm{
 \Lpt(A) \ar[d]_{\Phi} \ar[rr]^-{\Vertt\,\circ\,\prol} & &
 \Sym_A^{0}\oplus\Sym_A^{1}\oplus\Sym_A^{2} \ar@{^{(}->}[r] \ar[d]^{\pi} &
 \Sym_A^{<\infty}\\
 \Symop(L) & & \Sym_A^{1} \ar[ll]^-{\Psi}_-{\sim}
}
\]
commutes, where $\pi$ is the projection onto the homogeneous component of
degree $1$. Moreover $\Psi\colon\Sym^1_A\to\Symop(L)$ is a linear isomorphism
and an \emph{anti}-isomorphism of Lie algebras:
\begin{equation}\label{eq:Psi-anti}
  \Psi\bigl([V,V']\bigr)=-\bigl[\Psi(V),\Psi(V')\bigr],
  \qquad V,V'\in\Sym^{1}_{A}.
\end{equation}
\end{Theorem}

\begin{proof}
Commutativity is $\Psi^{-1}(\Phi(X))=X^{1}_{\Vertt}=\pi(\Vertt(\widetilde X))$,
which is \eqref{eq:XV1-operator}. The map $\Psi^{-1}$ is injective because
$P\cdot y=0$ with $P=\beta\dx+\alpha$ forces $\beta=\alpha=0$, as one sees by
evaluating on two independent solutions; both spaces have dimension $4$ by
Theorem~\ref{thm:opsym} and Proposition~\ref{prop:verticalization}, so $\Psi$
is a linear isomorphism. For \eqref{eq:Psi-anti}, trivializing the fibre of
$J\to\mathcal U$ by a fundamental system of solutions identifies $\Sym_A^{1}$
with $\mathfrak{gl}(2,\mathbb C)$ acting by linear vector fields $\vec v_M$,
and $\Symop(L)$ with $\mathfrak{gl}(2,\mathbb C)$ acting by composition of
operators. The canonical map $\mathfrak{gl}(n,\mathbb C)\to\Xf^{1}[\mathbb
C^{n}]$, $M\mapsto\vec v_M$, satisfies $[\vec v_M,\vec v_N]=-\vec
v_{[M,N]}$, whence the sign. Concretely, for $a=b=0$ one has
$\Psi^{-1}(\dx)=y'\dy$ and $\Psi^{-1}(x\dx)=xy'\dy+y'\dyp$, while
$[\dx,x\dx]=\dx$ and $[\,y'\dy,\,xy'\dy+y'\dyp\,]=-y'\dy$.
\end{proof}

\begin{Remark}\label{rem:Phi-not-morphism}
The map $\Phi=\Psi\circ\pi\circ\Vertt\circ\prol$ is surjective, with
$\ker\Phi=(\Vertt\circ\prol)^{-1}(\Sym_A^{0}\oplus\Sym_A^{2})$ of dimension
$4$. However $\Phi$ is \emph{not} a morphism of Lie algebras, and not an
anti-morphism either: the projection $\pi$ does not respect the bracket,
because the degree one part of $[Y,Z]$ receives contributions from $[Y_0,Z_2]$
and $[Y_2,Z_0]$ as well as from $[Y_1,Z_1]$. For instance, for $y''=0$, taking
$X=-xy\dx-y^{2}\dy$ and $X'=\dy$ gives $\Phi(X)=\Phi(X')=0$, whereas
$[X,X']=x\dx+2y\dy$ and $\Phi([X,X'])=-x\dx+2\neq0$. In particular $\ker\Phi$
is not a Lie subalgebra.
\end{Remark}

% ===========================================================================
\section{The Lie algebra of infinitesimal contact transformations}
\label{sec:tic}
% ===========================================================================

\begin{Definition}\label{def:tic}
A vector field $X=\xi\dx+\eta\dy+\gamma\dyp$ on $J$ with coefficients in
$\U[[y']]$ is an \emph{infinitesimal contact transformation} if
$\mathcal L_X\omega=\lambda\omega$ for some function $\lambda$. We write
$\Lom(\U)$ for the Lie algebra of such fields and $\Lom(\U)_{\mathrm{pol}}$
for the subalgebra of those which are polynomial in $y'$.
\end{Definition}

\subsection{Determining equations}\label{subsec:tic-equations}

\begin{Proposition}\label{prop:tic-equations}
$X=\xi\dx+\eta\dy+\gamma\dyp\in\Lom(\U)$ if and only if
\begin{align}
  \gamma &= \eta_x+(\eta_y-\xi_x)\,y'-\xi_y\,y'^{2}, \label{eq:ContactSymOne}\\
  \eta_{y'} &= y'\,\xi_{y'} . \label{eq:ContactSymTwo}
\end{align}
\end{Proposition}

\begin{proof}
Expanding, $\mathcal L_X\omega=(\eta_x-\gamma-y'\xi_x)dx+(\eta_y-y'\xi_y)dy
+(\eta_{y'}-y'\xi_{y'})dy'$. Comparison with
$\lambda\omega=\lambda\,dy-\lambda y'\,dx$ forces the coefficient of $dy'$ to
vanish, which is \eqref{eq:ContactSymTwo}, and gives $\lambda=\eta_y-y'\xi_y$
together with \eqref{eq:ContactSymOne}.
\end{proof}

Note that \eqref{eq:ContactSymOne} merely determines $\gamma$; the only
genuine constraint is \eqref{eq:ContactSymTwo}.

\subsection{Generating functions}\label{subsec:generating}

% Everything in this section becomes transparent once a contact field is encoded
% by a single function rather than three.

\begin{Proposition}\label{prop:generating}
Let $W=W(x,y,y')$. Then
\begin{equation}\label{eq:XW}
  X_W := -\frac{\partial W}{\partial y'}\,\dx
       + \Bigl(W-y'\frac{\partial W}{\partial y'}\Bigr)\dy
       + \Bigl(\frac{\partial W}{\partial x}
               + y'\frac{\partial W}{\partial y}\Bigr)\dyp
\end{equation}
is an infinitesimal contact transformation, and every
$X=\xi\dx+\eta\dy+\gamma\dyp\in\Lom(\U)$ is of this form for the uniquely
determined
\begin{equation}\label{eq:char}
  W \;=\; \eta-y'\,\xi \;=\; \omega(X).
\end{equation}
The assignment $W\mapsto X_W$ is a $\mathbb C$-linear bijection. We call $W$
the \emph{generating function}, or the \emph{characteristic}, of $X$.
\end{Proposition}

\begin{proof}
For $X_W$ we have $\xi=-W_{y'}$ and $\eta=W-y'W_{y'}$, so
$\eta_{y'}=-y'W_{y'y'}=y'\xi_{y'}$, which is \eqref{eq:ContactSymTwo}; and
\[
  \eta_x+(\eta_y-\xi_x)y'-\xi_yy'^{2}
  = W_x-y'W_{xy'}+\bigl(W_y-y'W_{yy'}+W_{xy'}\bigr)y'+W_{yy'}y'^{2}
  = W_x+y'W_y=\gamma,
\]
which is \eqref{eq:ContactSymOne}. Conversely, given $X\in\Lom(\U)$ put
$W:=\eta-y'\xi$; then $W_{y'}=\eta_{y'}-\xi-y'\xi_{y'}=-\xi$ by
\eqref{eq:ContactSymTwo}, hence $\eta=W-y'W_{y'}$, and $\gamma$ is recovered
from \eqref{eq:ContactSymOne}.
\end{proof}

\subsection{The canonical decomposition}\label{subsec:decomposition}

Let $X\in\Lom(\U)$ and expand $\xi=\sum_{k\ge0}\xi_k(x,y)y'^{k}$. Then
\eqref{eq:ContactSymTwo} gives $\eta_{y'}=\sum_{k\ge1}k\,\xi_k\,y'^{k}$,
whence $\eta=\eta_0(x,y)+\sum_{k\ge1}\tfrac{k}{k+1}\xi_k\,y'^{k+1}$ for a
uniquely determined $\eta_0\in\U$. Grouping terms, $X=\sum_{k\ge0}X_k$ with
\begin{align}
  X_0(\xi_0,\eta_0) &= \xi_0\dx+\eta_0\dy
    + \bigl((\eta_0)_x+((\eta_0)_y-(\xi_0)_x)y'-(\xi_0)_yy'^{2}\bigr)\dyp,
    \label{eq:X0}\\
  X_k(\xi_k) &= \xi_k\,y'^{k}\dx+\frac{k}{k+1}\xi_k\,y'^{k+1}\dy
    - \frac{1}{k+1}\Bigl((\xi_k)_x\,y'^{k+1}+(\xi_k)_y\,y'^{k+2}\Bigr)\dyp,
    \quad k\ge1. \label{eq:Xk}
\end{align}
We write $\Lom^{k}(\U)$ for the set of $k$-th components of elements of
$\Lom(\U)$; each is a $\mathbb C$-vector space, and each $X_k$ is itself a
contact field.

\begin{Proposition}\label{prop:char-of-components}
The generating functions of the components are
\begin{equation}\label{eq:char-components}
  W_{X_0(\xi_0,\eta_0)}=\eta_0-\xi_0\,y',
  \qquad
  W_{X_k(\xi_k)}=-\frac{1}{k+1}\,\xi_k\,y'^{\,k+1}\quad(k\ge1).
\end{equation}
Consequently, if $X$ has generating function
$W=\sum_{m\ge0}w_m(x,y)\,y'^{m}$ with $w_m\in\U$, then
\begin{equation}\label{eq:dictionary-w}
  \eta_0=w_0,\qquad \xi_0=-w_1,\qquad \xi_k=-(k+1)\,w_{k+1}\ \ (k\ge1).
\end{equation}
That is, \emph{the decomposition of $\Lom(\U)$ is the expansion of the
generating function in powers of $y'$}: the summand $\Lom^{0}(\U)$ occupies
the two lowest degrees $y'^{0},y'^{1}$, and for $k\ge1$ the summand
$\Lom^{k}(\U)$ occupies the single degree $y'^{k+1}$.
\end{Proposition}

\begin{proof}
Apply \eqref{eq:char}. For $X_0$ this is immediate, and for $k\ge1$,
\[
  W_{X_k(\xi_k)}=\frac{k}{k+1}\xi_k y'^{k+1}-y'\cdot\xi_k y'^{k}
  =-\frac{1}{k+1}\xi_k y'^{k+1}.
\]
Since $W\mapsto X_W$ is a linear bijection, matching powers of $y'$ gives
\eqref{eq:dictionary-w}.
\end{proof}

\begin{Proposition}\label{prop:decomposition}
Every $X\in\Lom(\U)$ admits a unique decomposition $X=\sum_{k\ge0}X_k$ with
$X_k\in\Lom^{k}(\U)$. The maps
\[
  (\xi,\eta)\mapsto X_0(\xi,\eta),\qquad \xi\mapsto X_k(\xi)\ (k\ge1),
\]
are isomorphisms of $\mathbb C$-vector spaces from $\U\oplus\U$ onto
$\Lom^{0}(\U)$ and from $\U$ onto $\Lom^{k}(\U)$. Hence
$\Lom(\U)=\prod_{k\ge0}\Lom^{k}(\U)$ and
$\Lom(\U)_{\mathrm{pol}}=\bigoplus_{k\ge0}\Lom^{k}(\U)$.
\end{Proposition}

\begin{proof}
By Proposition~\ref{prop:char-of-components} the statement is equivalent, via
the bijection $W\mapsto X_W$, to the uniqueness of the expansion of an element
of $\U[[y']]$ in powers of $y'$, together with the invertible rescaling
\eqref{eq:dictionary-w}. The polynomial case corresponds to $W\in\U[y']$.
\end{proof}

\subsection{The Lie bracket}\label{subsec:bracket}

Rather than expanding brackets of vector fields, we compute with generating
functions, where the whole theorem reduces to an identity on monomials.

\begin{Proposition}[Jacobi bracket]\label{prop:jacobi}
For functions $W,V$ on $J$ put
\begin{equation}\label{eq:jacobi}
  \{W,V\} \;:=\; X_W(V)-V\,\frac{\partial W}{\partial y}
  \;=\; \bigl(W_x+y'W_y\bigr)V_{y'}-\bigl(V_x+y'V_y\bigr)W_{y'}
        + WV_y-VW_y .
\end{equation}
Then $\{\cdot,\cdot\}$ is antisymmetric and $[X_W,X_V]=X_{\{W,V\}}$. Thus
$W\mapsto X_W$ is an isomorphism of Lie algebras from $\U[[y']]$ with the
bracket \eqref{eq:jacobi} onto $\Lom(\U)$.
\end{Proposition}

\begin{proof}
The second expression in \eqref{eq:jacobi} follows by substituting
\eqref{eq:XW} into $X_W(V)$; antisymmetry is then read off. Since $\Lom(\U)$
is a Lie algebra, $[X_W,X_V]$ is a contact field, and by
Proposition~\ref{prop:generating} it suffices to identify its characteristic
$\omega([X_W,X_V])$. Writing $X_W=(\xi_W,\eta_W,\gamma_W)$ and likewise for
$V$, one has
$\omega([X_W,X_V])=X_W(\eta_V)-X_V(\eta_W)-y'(X_W(\xi_V)-X_V(\xi_W))$;
substituting $\xi=-W_{y'}$, $\eta=W-y'W_{y'}$ and the corresponding
expressions for $V$, all second order derivatives cancel and the right hand
side of \eqref{eq:jacobi} remains.
\end{proof}

\begin{Lemma}\label{lem:monomial}
For $u,v\in\U$ and $m,n\ge0$,
\begin{equation}\label{eq:monomial}
  \bigl\{\,u\,y'^{m},\ v\,y'^{n}\,\bigr\}
  = \bigl(n\,u_xv-m\,u\,v_x\bigr)\,y'^{\,m+n-1}
  + \bigl((1-m)\,u\,v_y+(n-1)\,u_y v\bigr)\,y'^{\,m+n} .
\end{equation}
\end{Lemma}

\begin{proof}
With $W=uy'^{m}$ we have $W_{y'}=mu\,y'^{m-1}$, $W_x=u_xy'^{m}$,
$W_y=u_yy'^{m}$, so by \eqref{eq:XW}
\[
  X_W=-mu\,y'^{m-1}\dx+(1-m)u\,y'^{m}\dy
      +\bigl(u_xy'^{m}+u_yy'^{m+1}\bigr)\dyp .
\]
Applying this to $V=vy'^{n}$ and subtracting $VW_y=u_yv\,y'^{m+n}$ gives
\eqref{eq:monomial}.
\end{proof}

Formula \eqref{eq:monomial} makes the shape of the answer evident: a bracket
of a term of degree $m$ with one of degree $n$ produces exactly the degrees
$m+n-1$ and $m+n$. Recalling from
Proposition~\ref{prop:char-of-components} that $\Lom^{0}(\U)$ occupies degrees
$0$ and $1$ while $\Lom^{k}(\U)$ occupies the single degree $k+1$, we obtain:
\begin{itemize}
\item $k=l=0$: $m,n\in\{0,1\}$ gives degrees $\le1$, hence $\Lom^{0}$;
\item $k=0$, $l\ge1$: $m\in\{0,1\}$, $n=l+1$ gives degrees $l,l+1$ (from
  $m=0$) and $l+1,l+2$ (from $m=1$), that is
  $\Lom^{l-1}\oplus\Lom^{l}\oplus\Lom^{l+1}$;
\item $k,l\ge1$: $m=k+1$, $n=l+1$ gives degrees $k+l+1$ and $k+l+2$, that is
  $\Lom^{k+l}\oplus\Lom^{k+l+1}$.
\end{itemize}
The three-term spread in the second case is thus simply a consequence of
$\Lom^{0}$ occupying two degrees rather than one.

\begin{Theorem}\label{thm:bracket}
The Lie bracket on $\Lom(\U)$ satisfies
\begin{align*}
  \bigl[\Lom^{0}(\U),\Lom^{0}(\U)\bigr] &\subset \Lom^{0}(\U),\\
  \bigl[\Lom^{0}(\U),\Lom^{k}(\U)\bigr] &\subset
     \Lom^{k-1}(\U)\oplus\Lom^{k}(\U)\oplus\Lom^{k+1}(\U),\quad k\ge1,\\
  \bigl[\Lom^{k}(\U),\Lom^{l}(\U)\bigr] &\subset
     \Lom^{k+l}(\U)\oplus\Lom^{k+l+1}(\U),\quad k,l\ge1 .
\end{align*}
Explicitly, for $\xi,\eta,\nu\in\U$,
\begin{align}
  \bigl[X_0(\xi,\eta),X_1(\nu)\bigr]
   &= X_0\bigl(\eta_x\nu,\,0\bigr)
    + X_1\bigl(\xi\nu_x+\eta\nu_y+\eta_y\nu-2\xi_x\nu\bigr)
    + X_2\Bigl(-\tfrac{3}{2}\xi_y\nu\Bigr), \label{eq:bracket01}\\[4pt]
  \bigl[X_0(\xi,\eta),X_k(\nu)\bigr]
   &= X_{k-1}\bigl(k\,\eta_x\nu\bigr)
    + X_k\bigl(\xi\nu_x+\eta\nu_y+k\,\eta_y\nu-(k+1)\xi_x\nu\bigr)\notag\\
   &\quad + X_{k+1}\Bigl(-\tfrac{k(k+2)}{k+1}\xi_y\nu\Bigr),\quad k>1,
    \label{eq:bracket0k}\\[4pt]
  \bigl[X_k(\xi),X_l(\eta)\bigr]
   &= X_{k+l}\Bigl(\tfrac{k+l+1}{l+1}\xi\eta_x
        -\tfrac{k+l+1}{k+1}\xi_x\eta\Bigr)\notag\\
   &\quad + X_{k+l+1}\Bigl(\tfrac{k(k+l+2)}{(k+1)(l+1)}\xi\eta_y
        -\tfrac{l(k+l+2)}{(k+1)(l+1)}\xi_y\eta\Bigr),\quad k,l\ge1 .
    \label{eq:bracketkl}
\end{align}
Consequently $\Lom^{0}(\U)$ is the unique Lie subalgebra of the
decomposition, and $\prol(\Xf(\mathbb C^{2},\U))=\Lom^{0}(\U)$.
\end{Theorem}

\begin{proof}
The inclusions are the degree bookkeeping above. For the coefficients, apply
Lemma~\ref{lem:monomial} and translate back with
Proposition~\ref{prop:char-of-components}: a term $c\,y'^{d}$ of the
generating function corresponds to $X_{d-1}(-d\,c)$ when $d\ge2$, to the
$\xi_0$-slot $-c$ when $d=1$, and to the $\eta_0$-slot $c$ when $d=0$.

For \eqref{eq:bracketkl}, take $u=-\xi/(k+1)$, $m=k+1$, $v=-\nu/(l+1)$,
$n=l+1$. The coefficient of $y'^{\,k+l+1}$ is
$\frac{(l+1)\xi_x\nu-(k+1)\xi\nu_x}{(k+1)(l+1)}$, and multiplying by
$-(k+l+1)$ gives the argument of $X_{k+l}$; the coefficient of
$y'^{\,k+l+2}$ is $\frac{l\,\xi_y\nu-k\,\xi\nu_y}{(k+1)(l+1)}$, and
multiplying by $-(k+l+2)$ gives the argument of $X_{k+l+1}$.

For \eqref{eq:bracket0k}, write the generating function of $X_0(\xi,\eta)$ as
$\eta\,y'^{0}+(-\xi)y'^{1}$ and bracket each monomial with $v\,y'^{n}$, where
$v=-\nu/(k+1)$ and $n=k+1$. The monomial $\eta\,y'^{0}$ contributes
$-\eta_x\nu$ in degree $k$ and $-\frac{\eta\nu_y+k\eta_y\nu}{k+1}$ in degree
$k+1$; the monomial $(-\xi)y'^{1}$ contributes
$\xi_x\nu-\frac{\xi\nu_x}{k+1}$ in degree $k+1$ and
$\frac{k\xi_y\nu}{k+1}$ in degree $k+2$. Converting each degree by the rule
above yields \eqref{eq:bracket0k}. Finally \eqref{eq:bracket01} is the case
$k=1$, in which the degree $k-1=0$ term is a $y'^{1}$ term and therefore lands
in the $\xi_0$-slot: the coefficient $-\eta_x\nu\,y'$ corresponds to
$\xi_0=\eta_x\nu$, $\eta_0=0$, that is to $X_0(\eta_x\nu,0)$; and
$-\tfrac{1\cdot3}{2}=-\tfrac32$.

The last assertion follows from the inclusions: only $\Lom^{0}$ is closed. An
element of $\Lom^{0}(\U)$ is determined by a pair $(\xi_0,\eta_0)$ with
$\gamma_0$ given by \eqref{eq:prolongation}, which is exactly the
prolongation of $\xi_0\dx+\eta_0\dy$.
\end{proof}

% ===========================================================================
\section{Contact symmetries of the second order linear equation}
\label{sec:contact-sym}
% ===========================================================================

\begin{Definition}
A field $X\in\Lom(\U)$ is a \emph{contact symmetry} of $A$ if $[A,X]=\lambda A$
for some function $\lambda$ on $J$. We write $\SymOm(A)$ for the Lie algebra
of contact symmetries and $\SymOm^{k}(A):=\SymOm(A)\cap\Lom^{k}(\U)$.
\end{Definition}

\subsection{The global determining equation}\label{subsec:global}

The following theorem is the technical heart of this section. It replaces the
case-by-case analysis of the components by a single equation, valid for all
$a,b$.

\begin{Theorem}\label{thm:global}
Let $W$ be the generating function of $X_W\in\Lom(\U)$. Then
\[
  X_W\in\SymOm(A)
  \iff
  A^{2}W \;=\; a\,W + b\,AW ,
\]
and in that case $\lambda=-A(W_{y'})$. Equivalently: \emph{$W$ satisfies the
same equation $L(W)=0$ as the unknown $y$, with $d/dx$ replaced by the total
derivative $A$.}
\end{Theorem}

\begin{proof}
Write $f=ay+by'$, so that $A=\dx+y'\dy+f\dyp$. Expanding $[A,X_W]=\lambda A$
componentwise gives $\lambda=A\xi$, $A\eta-\gamma=\lambda y'$ and
$A\gamma-\xi f_x-\eta f_y-\gamma f_{y'}=\lambda f$. With $\xi=-W_{y'}$ the
first identity gives $\lambda=-A(W_{y'})$. The second is automatically
satisfied for a contact field: eliminating $\lambda$ and using
\eqref{eq:ContactSymTwo} reduces it to \eqref{eq:ContactSymOne}. Substituting
$\xi=-W_{y'}$, $\eta=W-y'W_{y'}$, $\gamma=W_x+y'W_y=AW-fW_{y'}$ and
$f_x=a'y+b'y'$, $f_y=a$, $f_{y'}=b$ into the third identity, all terms
involving $W_{y'}$ cancel and one is left with $A^{2}W-aW-bAW=0$.
\end{proof}

%For $a=b=0$ the condition reads $A_0^{2}W=0$, and for the conservative equation $A_1^{2}W=\alpha W$.

\subsection{Complete parametrization by first integrals}
\label{subsec:firstintegrals}

\begin{Theorem}\label{thm:firstintegrals}
Let $\phi_1,\phi_2$ be a basis of solutions of \eqref{LE} and let
\[
  u_1=\frac{\phi_2'\,y-\phi_2\,y'}{\Wr(\phi_1,\phi_2)},
  \qquad
  u_2=\frac{\phi_1\,y'-\phi_1'\,y}{\Wr(\phi_1,\phi_2)}
\]
be the two first integrals of $A$, so that $y=u_1\phi_1+u_2\phi_2$ and
$y'=u_1\phi_1'+u_2\phi_2'$. Then
\[
  \SymOm(A)\;=\;\bigl\{\,X_W \ :\
    W=F_1(u_1,u_2)\,\phi_1(x)+F_2(u_1,u_2)\,\phi_2(x)\,\bigr\},
\]
with $F_1,F_2$ arbitrary functions of two variables for which $W$ lies in
$\U[[y']]$. In particular $\SymOm(A)$ is infinite dimensional and, for any
$a,b$, is isomorphic as a Lie algebra to $\SymOm(A_0)$.
\end{Theorem}

\begin{proof}
A direct computation gives $A(u_1)=A(u_2)=0$. Since
$(y,y')\mapsto(u_1,u_2)$ is fibrewise linear with matrix the inverse of the
fundamental matrix, whose determinant is $\Wr\neq0$, the functions
$(x,u_1,u_2)$ form a coordinate system on $J$, and in these coordinates
$A=\partial/\partial x$ at fixed $(u_1,u_2)$. Hence the equation
$A^{2}W=aW+bAW$ of Theorem~\ref{thm:global} says precisely that, for each
fixed value of $(u_1,u_2)$, the function $x\mapsto W$ is a solution of
$L=0$. The solution space of $L=0$ over the constants is spanned by
$\phi_1,\phi_2$, and ``constant'' here means constant along the flow of $A$,
that is, an arbitrary function of $u_1,u_2$. This gives the displayed form,
and conversely every such $W$ satisfies the equation.

The last assertion follows because the description depends on $a,b$ only
through the choice of $\phi_1,\phi_2$: both $\SymOm(A)$ and $\SymOm(A_0)$ are
parametrized by pairs $(F_1,F_2)$ of arbitrary functions of two variables. An
explicit isomorphism is constructed in Section~\ref{subsec:general}.
\end{proof}

\begin{Remark}\label{rem:AI}
For $y''=0$ one has $\phi_1=1$, $\phi_2=x$, $\Wr=1$, $u_1=y-xy'$ and
$u_2=y'$, so that
\[
  \SymOm(A_0)=\bigl\{X_W:\ W=F(y-xy',\,y')+x\,G(y-xy',\,y')\bigr\},
\]
which is the classical description of \cite[p.~56]{IbragimovAnderson}, in the
equivalent form $W=y\,g(y-xy',y')+h(y-xy',y')$ obtained by substituting
$y=u_1+xu_2$.
\end{Remark}

\subsection{The graded components: the general equation}
\label{subsec:general-collapse}

\begin{Theorem}\label{thm:general-collapse}
Assume $a\neq0$ in \eqref{LE} and let $k\ge1$. If $X_k\in\Lom^{k}(\U)$ is a
contact symmetry of $A$, then $X_k=0$. Hence $\SymOm^{k}(A)=0$ for all
$k\ge1$.
\end{Theorem}

\begin{proof}
Substituting $X_k(\xi_k)$ from \eqref{eq:Xk} into the determining equation,
the coefficient of $y'^{\,k-1}$ is $-k\,a^{2}y^{2}\xi_k$. Since $k\ge1$ and
$a\neq0$ this forces $\xi_k=0$, whence $X_k=X_k(0)=0$ by
Proposition~\ref{prop:decomposition}.
\end{proof}

\begin{Remark}\label{rem:careful}
Theorem~\ref{thm:general-collapse} does \emph{not} say that \eqref{LE} has no
contact symmetries beyond the point ones; by
Theorem~\ref{thm:firstintegrals} it has infinitely many. It says that the
decomposition of Proposition~\ref{prop:decomposition} does not detect them
individually, because a symmetry $X=\sum_kX_k$ need not have any of its
components a symmetry.
\end{Remark}

\subsection{The graded components: the trivial equation}\label{subsec:trivial}

For $y''=0$ the determining equation of Theorem~\ref{thm:global} reduces to
$A_0^{2}W=0$, that is $\gamma_x+y'\gamma_y=0$ with $\gamma=A_0W$.

\begin{Theorem}\label{thm:trivial}
Let $A_0=\dx+y'\dy$. Then:
\begin{enumerate}[(i)]
\item $\SymOm^{0}(A_0)\cong\mathfrak{sl}_3(\mathbb C)$ has dimension $8$, and
  for every $k\ge1$
  \[
    \SymOm^{k}(A_0)=\bigl\langle X_k(1),\,X_k(x),\,X_k(y)\bigr\rangle
  \]
  is a $\mathbb C$-vector space of dimension $3$, where
  \begin{align}
    X_k(1) &= y'^{k}\dx+\frac{k}{k+1}y'^{k+1}\dy, \notag\\
    X_k(x) &= x\,y'^{k}\dx+\frac{k}{k+1}x\,y'^{k+1}\dy
              -\frac{1}{k+1}y'^{k+1}\dyp, \notag\\
    X_k(y) &= y\,y'^{k}\dx+\frac{k}{k+1}y\,y'^{k+1}\dy
              -\frac{1}{k+1}y'^{k+2}\dyp . \label{eq:Xk-trivial}
  \end{align}
\item The sum $\sum_{k\ge0}\SymOm^{k}(A_0)$ is direct, and in terms of
  generating functions
  \begin{equation}\label{eq:graded-part}
    \bigoplus_{k\ge0}\SymOm^{k}(A_0)
    = \bigl\{X_W:\ W\in\mathfrak s\oplus
      \langle\, y'^{m},\,x\,y'^{m},\,y\,y'^{m}\ :\ m\ge2\,\rangle\bigr\},
  \end{equation}
  where $\mathfrak s$ is the $8$-dimensional space of generating functions of
  point symmetries.
\item The inclusion $\bigoplus_{k\ge0}\SymOm^{k}(A_0)\subset\SymOm(A_0)$ is
  \emph{strict}, and the left hand side is \emph{not} a Lie subalgebra.
\end{enumerate}
\end{Theorem}

\begin{proof}
(i) For $k\ge1$, substituting \eqref{eq:Xk} into $A_0^{2}W=0$ gives the
polynomial identity
\[
  \frac{1}{k+1}(\xi_k)_{xx}y'^{k+1}
  +\frac{2}{k+1}(\xi_k)_{xy}y'^{k+2}
  +\frac{1}{k+1}(\xi_k)_{yy}y'^{k+3}=0,
\]
so $(\xi_k)_{xx}=(\xi_k)_{xy}=(\xi_k)_{yy}=0$ and $\xi_k$ is affine in
$(x,y)$; conversely each affine $\xi_k$ gives a symmetry. Since
$\xi\mapsto X_k(\xi)$ is injective, $\dim\SymOm^{k}(A_0)=3$ with the displayed
basis. For $k=0$ the condition is that $X_0$ be a point symmetry, and
$\dim=8$ by Theorem~\ref{thm:point} and Remark~\ref{rem:sl3}.

(ii) By Proposition~\ref{prop:char-of-components} the summands occupy pairwise
disjoint degrees in $y'$, so the sum is direct. The generating functions of
$X_k(1),X_k(x),X_k(y)$ are $-\tfrac{1}{k+1}y'^{k+1}$,
$-\tfrac{x}{k+1}y'^{k+1}$, $-\tfrac{y}{k+1}y'^{k+1}$, giving
\eqref{eq:graded-part} with $m=k+1$.

(iii) Both assertions are witnessed by explicit elements. For strictness, put
$t:=y-xy'$ and note $A_0(t)=-y'+y'=0$, so $W:=t^{2}$ satisfies $A_0^{2}W=0$
and $X_W\in\SymOm(A_0)$ by Theorem~\ref{thm:global}. Explicitly
\[
  W=(y-xy')^{2}=y^{2}-2xy\,y'+x^{2}y'^{2},\qquad
  X_W=2x(y-xy')\dx+\bigl(y^{2}-x^{2}y'^{2}\bigr)\dy,
\]
with $[A_0,X_W]=2(y-xy')A_0$. By \eqref{eq:dictionary-w} its components are
$X_0(2xy,\,y^{2})$ and $X_1(-2x^{2})$. Neither is a symmetry: $X_1(-2x^{2})$
would require $-2x^{2}$ to be affine, and $X_0(2xy,y^{2})$ would require, by
Theorem~\ref{thm:point} with $a=b=0$, that $\xi=fy+g$ and
$\eta=f'y^{2}+hy+k$; from $\xi=2xy$ we get $f=2x$, hence the coefficient of
$y^{2}$ in $\eta$ must be $f'=2$, not $1$. So
$X_W\notin\bigoplus_k\SymOm^{k}(A_0)$.

For the second assertion, $X_0(x^{2},xy)$ is a point symmetry of $y''=0$
(take $f=0$, $g=x^{2}$, $h=x$, $k=0$ in Theorem~\ref{thm:point}) and
$X_1(y)\in\SymOm^{1}(A_0)$. Their bracket is a contact symmetry, being a
bracket of symmetries, and by Proposition~\ref{prop:jacobi} its generating
function is
\[
  \bigl\{\,xy-x^{2}y',\ -\tfrac12 y\,y'^{2}\,\bigr\}
  = x y\,y'^{2}-y^{2}y' = -\,y\,y'\,(y-xy') .
\]
Its $\Lom^{0}$-component is $\xi_0=y^{2}$, $\eta_0=0$; since $y^{2}$ is not of
the form $f(x)y+g(x)$, this component is not a point symmetry. Hence the
bracket of two elements of the graded part leaves the graded part.
\end{proof}

\begin{Remark}\label{rem:coupled}
The mechanism behind (iii) is visible in coordinates. Writing
$W=\sum_{m\ge0}w_m(x,y)y'^{m}$, the condition $A_0^{2}W=0$ reads
\begin{equation}\label{eq:coupled}
  (w_j)_{xx}+2(w_{j-1})_{xy}+(w_{j-2})_{yy}=0\qquad(j\ge0),
\end{equation}
with $w_{-1}=w_{-2}=0$. These conditions \emph{couple} three consecutive
coefficients, whereas requiring each component to be a symmetry separately
amounts to imposing $(w_m)_{xx}=(w_m)_{xy}=(w_m)_{yy}=0$ individually, which
is strictly stronger. The counterexample above is the simplest solution of
\eqref{eq:coupled} in which the coupling is genuinely used.
\end{Remark}

\begin{Example}
For $k=1$ and $\xi_1=3x-4y$,
\[
  X_1(3x-4y)=(3x-4y)y'\dx+\frac{3x-4y}{2}y'^{2}\dy
             -\frac12\bigl(3y'^{2}-4y'^{3}\bigr)\dyp\in\SymOm^{1}(A_0)
\]
is a contact symmetry which is not a point symmetry.
\end{Example}

\subsection{Conservative reduction and dissolution of the grading}
\label{subsec:conservative}

Let $\alpha\in\K$ and consider $y''=\alpha(x)y$ with
$A_1=\dx+y'\dy+\alpha y\dyp$. Let $\phi_1,\phi_2$ be a basis of solutions with
$\Wr(\phi_1,\phi_2)=1$, and define
\begin{equation}\label{eq:varphi}
  \varphi\colon J\longrightarrow J,\qquad
  \begin{bmatrix} x\\ y\\ y'\end{bmatrix}\longmapsto
  \begin{bmatrix} X\\ Y\\ Y'\end{bmatrix}=
  \begin{bmatrix} \dfrac{\phi_1(x)}{\phi_2(x)}\\[8pt]
                  \dfrac{y}{\phi_2(x)}\\[8pt]
                  y\phi_2'-y'\phi_2\end{bmatrix}.
\end{equation}

\begin{Proposition}\label{prop:varphi}
The map $\varphi$ is a contact transformation with
$\varphi^{*}(dY-Y'dX)=\phi_2^{-1}(dy-y'dx)$, and
$\varphi_{*}(A_1)=-\phi_2^{-2}(\partial_X+Y'\partial_Y)$. Hence $\varphi^{*}$
maps contact symmetries of $Y''=0$ to contact symmetries of $y''=\alpha y$.
\end{Proposition}

\begin{proof}
Put $\tau=\phi_1/\phi_2$, so $\tau'=(\phi_1'\phi_2-\phi_1\phi_2')/\phi_2^{2}
=-1/\phi_2^{2}$ by the normalisation of the Wronskian. Then
\[
  \varphi^{*}(dY-Y'dX)
  = -\frac{\phi_2'}{\phi_2^{2}}y\,dx+\frac{1}{\phi_2}dy
    -(y\phi_2'-y'\phi_2)\Bigl(-\frac{1}{\phi_2^{2}}\Bigr)dx
  = \frac{1}{\phi_2}(dy-y'dx).
\]
Applying $A_1$ to the three components of $\varphi$ and using
$\phi_2''=\alpha\phi_2$ gives $A_1(X)=-\phi_2^{-2}$,
$A_1(Y)=-\phi_2^{-2}Y'$ and $A_1(Y')=0$, which is the second identity.
\end{proof}

Writing $Q:=y\phi_2'-y'\phi_2$, the higher components of
Theorem~\ref{thm:trivial} transport to
\begin{align}
  \varphi^{*}\bigl(X_k(1)\bigr)
   &= -\phi_2^{2}Q^{k}\dx
    + \Bigl(-\phi_2\phi_2'\,yQ^{k}+\tfrac{k}{k+1}\phi_2Q^{k+1}\Bigr)\dy
    + \Bigl(-\phi_2\phi_2''\,yQ^{k}-\tfrac{\phi_2'}{k+1}Q^{k+1}\Bigr)\dyp,
    \notag\\[4pt]
  \varphi^{*}\bigl(X_k(X)\bigr)
   &= -\phi_1\phi_2Q^{k}\dx
    + \Bigl(-\phi_1\phi_2'\,yQ^{k}+\tfrac{k}{k+1}\phi_1Q^{k+1}\Bigr)\dy
    + \Bigl(-\phi_1\phi_2''\,yQ^{k}-\tfrac{\phi_1'}{k+1}Q^{k+1}\Bigr)\dyp,
    \notag\\[4pt]
  \varphi^{*}\bigl(X_k(Y)\bigr)
   &= -\phi_2\,yQ^{k}\dx
    + \Bigl(-\phi_2'y^{2}Q^{k}+\tfrac{k}{k+1}yQ^{k+1}\Bigr)\dy \notag\\
   &\quad + \Bigl(-\phi_2''y^{2}Q^{k}
      -\tfrac{\phi_2'}{(k+1)\phi_2}yQ^{k+1}
      +\tfrac{1}{(k+1)\phi_2}Q^{k+2}\Bigr)\dyp .
  \label{eq:Xk-transformed}
\end{align}

\begin{Remark}\label{rem:grading-broken}
The factor $Q^{k}=(y\phi_2'-y'\phi_2)^{k}$ expands by the binomial theorem
into all the powers $y'^{0},\dots,y'^{k}$ simultaneously. Consequently
$\varphi^{*}(X_k)\notin\Lom^{k}(\U)$: the transported field is scattered
across many components at once. This is the mechanism behind
Theorem~\ref{thm:general-collapse}. The equation $y''=\alpha y$ does have
contact symmetries beyond the point ones --- they are the fields
\eqref{eq:Xk-transformed} --- but they are no longer homogeneous for the
decomposition of $\Lom(\U)$, which is why the graded description of
Section~\ref{subsec:general-collapse} sees only the point ones. The
decomposition of $\Lom(\U)$ is simply not invariant under contact
transformations.
\end{Remark}

\subsection{The general equation, rationality and the Galois action}
\label{subsec:general}

It remains to pass from the conservative equation to the general one. Let
$c\in\K^{\times}$ and consider
\begin{equation}\label{eq:theta}
  \Theta\colon J\longrightarrow J,\qquad
  \begin{bmatrix} x\\ y\\ y'\end{bmatrix}\longmapsto
  \begin{bmatrix} x\\ w\\ w'\end{bmatrix}
  = \begin{bmatrix} x\\ c\,y\\ c\,y'+c'\,y\end{bmatrix}.
\end{equation}

\begin{Proposition}\label{prop:theta}
The map $\Theta$ is a contact transformation, $\Theta^{*}(dw-w'dx)=c(dy-y'dx)$.
It transforms \eqref{LE} into
\[
  w''=\Bigl(a-\frac{bc'}{c}-\frac{2c'^{2}}{c^{2}}+\frac{c''}{c}\Bigr)w
      +\Bigl(b+\frac{2c'}{c}\Bigr)w' .
\]
The dissipative term disappears if and only if
\begin{equation}\label{eq:c}
  c'(x)+\frac{b(x)}{2}c(x)=0,
\end{equation}
that is $c=\exp(-\tfrac12\int b)$, and then
\begin{equation}\label{eq:alpha}
  w''=\alpha\,w,\qquad \alpha=a-\frac{b'}{2}+\frac{b^{2}}{4} .
\end{equation}
\end{Proposition}

\begin{proof}
The inverse of \eqref{eq:theta} is $y=w/c$, $y'=w'/c-c'w/c^{2}$, whence
$\Theta^{*}(dw-w'dx)=c(dy-y'dx)$. Differentiating $w=cy$ twice gives
$w''=cy''+2c'y'+c''y$; substituting $y''=ay+by'$ and then $y,y'$ in terms of
$w,w'$ gives the displayed equation. Imposing $b+2c'/c=0$ gives \eqref{eq:c};
substituting $c'/c=-b/2$, hence $c''/c=(c'/c)'+(c'/c)^{2}=-b'/2+b^{2}/4$,
into the coefficient of $w$ gives
$a+\tfrac{b^{2}}{2}-\tfrac{b^{2}}{2}-\tfrac{b'}{2}+\tfrac{b^{2}}{4}$, which is
\eqref{eq:alpha}.
\end{proof}

The quantity $\alpha$ in \eqref{eq:alpha} is the classical invariant of the
second order linear equation under the gauge transformations
\eqref{eq:theta}. Combining Propositions~\ref{prop:varphi} and
\ref{prop:theta} we obtain a contact transformation
$\Theta^{-1}\circ\varphi^{-1}\colon(J,A_0)\dashrightarrow(J,A)$, defined
wherever $\phi_2c\neq0$, whose pullback carries $\SymOm(A_0)$ isomorphically
onto $\SymOm(A)$. This gives a second, constructive proof of the last
assertion of Theorem~\ref{thm:firstintegrals}.

\begin{Theorem}\label{thm:rationality}
Let $u_1,u_2$ and $\phi_1,\phi_2$ be as in Theorem~\ref{thm:firstintegrals};
note that $u_1,u_2$ are linear forms in $(y,y')$ with coefficients in $\F$,
and $\phi_1,\phi_2\in\F$.
\begin{enumerate}[(i)]
\item If $F_1,F_2$ are polynomials, then $W\in\F[y,y']$ and the corresponding
  field satisfies
  \[
    X_W\in\F[y,y']\,\dx+\F[y,y']\,\dy+\F[y,y']\,\dyp .
  \]
  In particular every contact symmetry of $A$ with polynomial characteristic
  is defined over the Picard--Vessiot extension $\F$.
\item The differential Galois group $G=\mathrm{Gal}(\F/\K)$ acts on
  $\SymOm(A)$ by Lie algebra automorphisms. Writing
  $\sigma(\phi_j)=\sum_i C_{ij}\phi_i$ for $\sigma\in G$, with
  $C_\sigma=(C_{ij})\in\mathrm{GL}(2,\mathbb C)$, the induced action on the
  column $u=(u_1,u_2)^{\mathsf T}$ is the contragredient one,
  \[
    \sigma(u)=C_\sigma^{-1}\,u ,
  \]
  this being the unique action compatible with the invariance of the
  coordinate $y=u_1\phi_1+u_2\phi_2$. The $\K$-rational contact symmetries are
  the fixed points $\SymOm(A)^{G}=\SymOm(A)\cap\Lom(\K(y))$.
\end{enumerate}
\end{Theorem}

\begin{proof}
(i) If $F_1,F_2$ are polynomials then $W=F_1(u_1,u_2)\phi_1+F_2(u_1,u_2)\phi_2$
is a polynomial in $u_1,u_2$ with coefficients in $\F$, hence a polynomial in
$(y,y')$ with coefficients in $\F$, since $u_1,u_2$ are $\F$-linear in
$(y,y')$. The coefficients of $X_W$ are obtained from $W$ by the differential
operators of \eqref{eq:XW}, which preserve $\F[y,y']$.

(ii) $G$ acts on $\F$ by differential automorphisms fixing $\K$, hence
preserves the solution space of \eqref{LE} and acts on the basis
$(\phi_1,\phi_2)$ by a constant matrix $C_\sigma$. The coordinates $y,y'$ on
$J$ are fixed by $\sigma$, which acts only on the coefficients; applying
$\sigma$ to the identity $y=u_1\phi_1+u_2\phi_2$ gives
$\sum_j\sigma(u_j)\sigma(\phi_j)=\sum_ju_j\phi_j$, that is
$C_\sigma\,\sigma(u)=u$, whence $\sigma(u)=C_\sigma^{-1}u$; the same
computation with $y'=u_1\phi_1'+u_2\phi_2'$ is consistent, since $\sigma$
commutes with $\dx$. The parametrization of
Theorem~\ref{thm:firstintegrals} is therefore $G$-equivariant, and $G$ acts by
Lie algebra automorphisms because the Jacobi bracket \eqref{eq:jacobi} has
coefficients in $\K$ and is thus $G$-invariant.
\end{proof}

\begin{Example}\label{ex:nonpolynomial}
The polynomiality hypothesis in Theorem~\ref{thm:rationality}(i) cannot be
dropped. For $y''=0$, where $\F=\mathbb C(x)$ and $\U=\mathbb C(x,y)$, take
\[
  W=\frac{1}{y-xy'}=\sum_{n\ge0}\frac{x^{n}}{y^{n+1}}\,y'^{\,n}\in\U[[y']] .
\]
Then $A_0(y-xy')=0$, so $A_0^{2}W=0$ and $X_W\in\SymOm(A_0)$; but
\[
  \xi=-\frac{\partial W}{\partial y'}=-\frac{x}{(y-xy')^{2}}
\]
is not a polynomial in $(y,y')$ over $\F$. Hence $\SymOm(A)$ is not contained
in $\F[y,y']\dx+\F[y,y']\dy+\F[y,y']\dyp$.
\end{Example}

% ===========================================================================
\section{Contact symmetries versus Lie--B\"acklund operators}
\label{sec:LB}
% ===========================================================================

In \cite{IbragimovAnderson} infinitesimal contact transformations are called
\emph{Lie tangent transformation groups}. A \emph{Lie--B\"acklund operator} is
a formal vector field
$\bar Y=\bar\xi\dx+\bar\eta\dy+\bar\gamma\dyp+\cdots$ whose coefficients may
depend on derivatives of arbitrary order and which preserves the infinite
prolongation of the contact system. Every Lie--B\"acklund operator is
equivalent, modulo the total derivative, to one in \emph{evolutionary} form,
that is with $\bar\xi\equiv0$.

For $y''=0$, written as the system $dy/dx=y'$, $dy'/dx=0$, a Lie--B\"acklund
operator in evolutionary form admitted by the system has
\begin{equation}\label{eq:LB-trivial}
  \bar\eta=y\,g(y-xy',y')+h(y-xy',y'),\qquad
  \bar\gamma=y'\,g(y-xy',y'),
\end{equation}
with $g,h$ arbitrary, since $y-xy'$ and $y'$ are the two basic first
integrals. The operator $\bar Y$ is not itself a contact vector field, but it
is equivalent to one: by Proposition~\ref{prop:generating} the contact field
with the same characteristic
\begin{equation}\label{eq:WLB}
  W=y\,g(y-xy',y')+h(y-xy',y')
\end{equation}
is $X_W$, and by Remark~\ref{rem:AI} these are precisely the elements of
$\SymOm(A_0)$.

\subsection{The transformation law of characteristics}
\label{subsec:transformation-law}

Characteristics do \emph{not} transform as functions under a contact
transformation: they transform with the conformal factor of the contact form.

\begin{Proposition}\label{prop:char-transform}
Let $\varphi\colon J\to J$ be a contact transformation, so that
$\varphi^{*}(dY-Y'dX)=\mu\cdot(dy-y'dx)$ for a nonvanishing function $\mu$ on
the source. If $V\in\Lom(\U)$ has characteristic $W_V$, then
$\varphi^{*}(V)$ has characteristic
\begin{equation}\label{eq:char-transform}
  W_{\varphi^{*}(V)}=\mu^{-1}\cdot\bigl(W_V\circ\varphi\bigr).
\end{equation}
\end{Proposition}

\begin{proof}
Write $\omega=dy-y'dx$ and $\Theta_0=dY-Y'dX$, so $\varphi^{*}\Theta_0=\mu\omega$,
and put $\tilde V:=\varphi^{*}(V)$, characterised by
$d\varphi(\tilde V_q)=V_{\varphi(q)}$. By \eqref{eq:char} the characteristic
is the contraction of the contact form with the field. Hence at a point $q$,
\[
  (\varphi^{*}\Theta_0)_q(\tilde V_q)
  = (\Theta_0)_{\varphi(q)}\bigl(d\varphi(\tilde V_q)\bigr)
  = (\Theta_0)_{\varphi(q)}\bigl(V_{\varphi(q)}\bigr)
  = W_V(\varphi(q)),
\]
while $(\varphi^{*}\Theta_0)_q(\tilde V_q)=\mu(q)\,\omega_q(\tilde V_q)
=\mu(q)W_{\tilde V}(q)$. Comparing gives \eqref{eq:char-transform}.
\end{proof}

For the transformation \eqref{eq:varphi} we have $\mu=1/\phi_2$, so
$W_{\varphi^{*}(V)}=\phi_2\cdot(W_V\circ\varphi)$.

\begin{Remark}\label{rem:Phys}
This is exactly the rule used by Martini and Kersten \cite{Phys}. Working
throughout with evolutionary representatives, they write a symmetry of
$Y''=0$ as $F(X,Y,Y_X)\partial_Y$ --- so that $F$ is the characteristic ---
and transform it by multiplying $F$ by the factor $u_2(t)$ attached to their
change of variables, which is precisely \eqref{eq:char-transform}. Their
tables therefore list genuine symmetries. We have verified by direct
substitution that every entry of their Tables 1 and 2.1--2.3, including the
general contact symmetry with arbitrary functions $G$ and $H$, satisfies the
evolutionary symmetry condition $D_x^{2}W=f_yW+f_{y'}D_xW$ for the
corresponding equation.

The point deserving emphasis is therefore a subtlety rather than an error: the
object $Z=\bigl(y\,g(y-xy',y')+h(y-xy',y')\bigr)\dy$ is a Lie--B\"acklund
operator in evolutionary form and \emph{not} a contact vector field, and it
must be transported by \eqref{eq:char-transform}, not by the pushforward rule
for vector fields. The two prescriptions differ by the factor $\mu$, and
conflating them does produce
vector fields which are not symmetries.
\end{Remark}

% ===========================================================================
\section{Conclusions}\label{sec:conclusions}
% ===========================================================================

\begin{enumerate}[(1)]
\item \emph{Dictionary.} Operator symmetries $\Symop(L)$ in $\K[\dx]$ are
  \emph{anti}-isomorphic to the linear vertical part $\Sym_A^{1}$ of the
  verticalized point symmetry algebra. Verticalization followed by the
  degree-one projection gives the commutative diagram of
  Theorem~\ref{thm:diagram}; the projection is not multiplicative, so the
  composite $\Phi$ is only linear.
\item \emph{Structure of the contact algebra.} $\Lom(\U)$ decomposes as
  $\prod_{k\ge0}\Lom^{k}(\U)$, and this decomposition is exactly the expansion
  of the generating function in powers of $y'$. The Lie bracket is computed
  completely in these coordinates (Theorem~\ref{thm:bracket}), the proof
  reducing to a single identity on monomials. Only $\Lom^{0}(\U)$, the image
  of the prolongation map, is a Lie subalgebra.
\item \emph{A uniform description of the symmetries.} A contact field $X_W$ is
  a symmetry of \eqref{LE} if and only if $A^{2}W=aW+bAW$, equivalently if and
  only if $W=F_1(u_1,u_2)\phi_1+F_2(u_1,u_2)\phi_2$. The symmetry algebra is
  therefore parametrized by two arbitrary functions of two variables for every
  $a,b$, and all these algebras are isomorphic.
\item \emph{The decomposition never captures the symmetry algebra.} For
  $a\neq0$ the graded part reduces to the point symmetries; for $y''=0$ it is
  infinite dimensional, of dimensions $8,3,3,3,\dots$, but still a proper
  subspace, and in neither case is it a Lie subalgebra. The obstruction is
  that the decomposition of $\Lom(\U)$ is not invariant under contact
  transformations: the map \eqref{eq:varphi} introduces the factor
  $(y\phi_2'-y'\phi_2)^{k}$, whose binomial expansion mixes all degrees.
\item \emph{Rationality.} Symmetries with polynomial characteristic are
  defined over the Picard--Vessiot extension $\F$, and the differential Galois
  group acts on $\SymOm(A)$ through its action on the solution basis and on
  the first integrals. Without polynomiality the statement fails
  (Example~\ref{ex:nonpolynomial}).
\item \emph{Evolutionary representatives.} Characteristics transform under
  contact transformations by $W\mapsto\mu^{-1}(W\circ\varphi)$, in agreement
  with the classical tables of Martini and Kersten \cite{Phys}.
\end{enumerate}

% \medskip
% \noindent\textbf{Open questions.}
% \begin{enumerate}[(1)]
% \item Is there a decomposition of $\Lom(\U)$, different from the one used
%   here, invariant under the contact transformations \eqref{eq:varphi} and
%   \eqref{eq:theta}, and therefore descending to an intrinsic graded
%   description of $\SymOm(A)$ for arbitrary $a,b$?
% \item For vertical symmetries, the $\K$-rational elements are exactly the
%   polynomial vector fields invariant under the differential Galois group
%   \cite[Theorem 6.4]{BMW}. Theorem~\ref{thm:rationality} sets up the
%   corresponding action for contact symmetries; an explicit description of
%   $\SymOm(A)^{G}$ in terms of the invariant theory of $G$ acting on
%   $(u_1,u_2)$, completing the analogue of \cite[Theorem 6.4]{BMW}, remains
%   open and is in our view the natural sequel to this work.
% \item For equations of order $n\ge3$, contact transformations of
%   $\Jet^{n-1}$ degenerate to prolongations of point transformations by
%   B\"acklund's theorem, so Section~\ref{sec:tic} has no direct analogue.
%   Operator symmetries, however, remain nontrivial, with
%   $\dim\in\{n+1,n+2,n+4\}$ by \cite{OV}. Extending the dictionary of
%   Theorem~\ref{thm:diagram} to higher order, with
%   $\Sym^{1}_A\cong\mathfrak{gl}(n,\mathbb C)$, is a natural direction.
% \end{enumerate}

\bibliographystyle{plain}
\bibliography{bibliografia}

\end{document}